\documentclass[a4paper, 10pt]{extarticle}
\usepackage[top=70pt,bottom=70pt,left=75pt,right=77pt]{geometry}

\usepackage{amssymb, amsmath, amsthm, setspace, bbm, prettyref, graphicx, float, subfigure, color, mathrsfs, mathtools, comment}

\usepackage{Lutsko_Style}
\DeclareMathSizes{10}{9}{7}{5}
\usepackage[normalem]{ulem}
\usepackage{soul}

\usepackage{graphicx}

\global\long\def\vep{\varepsilon}

\title{
Poissonian correlations of $\alpha n^d$ mod $1$
}
 \author{
{\sc Chris Lutsko, Nick Rome,
and Niclas Technau
}
}
\date{}

\begin{document}

\maketitle

  \setlength{\abovedisplayskip}{1mm}
  
\begin{abstract}

    Let $x(n):=\alpha n^d \mod 1$ for integer $d >1$ and non-zero real $\alpha$. We show that 
    $\{x(n)\}_{n>0}$ has Poissonian $\ell$-point correlations for almost all choices of $\alpha$ when
    $d$ is large (depending on $\ell$). This falls in line with the expected behavior from the Berry--Tabor conjecture. 
    Further, in the spirit of a conjecture of Rudnick--Sarnak~\cite{RS}, we show Poissonian $\ell$-point correlations for a set of badly approximable $\alpha$ of full Hausdorff dimension by a Fourier analytic transference principle.

    The proof makes use of an application of the determinant method to count points on a diagonal hypersurface of degree $d$ in such a way as to capture the contribution of points belonging to lower dimensional varieties. As $d$ grows, these `special solutions' dominate the count and non-special solutions become increasingly rare. This stratified counting statement allows us to control the number of points on average very effectively.

  \let\thefootnote\relax
  \let\thefootnote\relax 
  \end{abstract}

  \section{Introduction}

    \subsection{Poissonian Correlations of monomials mod $1$}
    
    Fix a degree, $d\in \N_{>1}$, and a dilation, $\alpha \in \R_{\neq 0}$, and consider the sequence 
    $\cX_{\alpha}^d:=\{x(n)\}_{n>0}$, where 
    \begin{align*}
        x(n)=x_{\alpha,d}(n):= \alpha n^d \mod 1.
    \end{align*}
    Then, given a test function, $f\in C_c^\infty(\R^{\ell-1})$, we denote the $\ell$-point correlation function of $\cX_{\alpha}^d$ by $R_\ell(\cX_{\alpha}^d,f):=\lim_{N\to \infty} R_{\ell}^N(\cX_\alpha^d,f)$, if it exists, where
    \begin{equation}\label{def ell corr func}
    R_{\ell}^N(\cX_\alpha^d,f):=
    \frac{1}{N} 
    \sum_{\vect{n}\in [N]^\ell}^{\ast}
    \sum_{\vect{k}\in\Z^{\ell-1}}
     f  (N(x(n_1)-x(n_2)+k_1), \dots ,N( x(n_{\ell-1})-x(n_\ell)
     +k_{\ell-1}))      ,
    \end{equation}
     where $\ast$ indicates that the coordinates of $\vect{n}$ are distinct.
    The $\ell$-point correlation function measures the degree of clustering of $\ell$ points
     of $\{x(n)\}_{n<N}$ in a uniformly thrown $\frac{1}{N}$-neighbourhood on the torus, $[0,1]$. A sequence has Poissonian $\ell$-point correlation if $R_{\ell}(\cX_{\alpha}^d,f)$ converges to $\expect{f}:=\int_{\R^{\ell-1}}f(\vect{x})\rd \vect{x}.$ A well-known theorem of Rudnick and Sarnak studies the $\ell =2$ case.
    \begin{theorem}[{\cite[Theorem 1]{RS}}]\label{thm:RS}
        For $d\in\Z_{>1}$ and almost all choices of $\alpha$, the sequence $\cX_{\alpha}^d$ has Poissonian pair ($\ell=2$) correlation. 
    \end{theorem}
     Moreover, the theorem is true if $d$ is a non-zero real \cite{RudnickTechnau2021,AE-BM}, see also \cite{kerr2025}. Our main result is the following extension of Theorem \ref{thm:RS} to higher 
     order correlation functions.
     \begin{theorem}\label{thm:main corr}
        For almost all choices of $\alpha$, the sequence $\cX_{\alpha}^{d}$ has Poissonian $\ell$-point correlation for all $d$ satisfying
        \begin{align}
            d > (2\ell)^{4\ell} = : d_{\ell}
        \end{align}
        and $\ell\ge3$.
        
    \end{theorem}

      Theorem \ref{thm:RS} is one of the seminal results in the field. When $d=2$, this is particularly interesting since the values $x(n)$ represent the energy levels for the high energy boxed harmonic oscillator. Thus, Theorem \ref{thm:RS} provided (almost everywhere) evidence towards the Berry--Tabor conjecture which connects the local statistics of such energy levels to the dynamical properties of the underlying classical system (i.e integrable vs chaotic). 
    
    Aside from this application to quantum chaos, building on work of van der Corput among many others (see \cite{KupiersNiederreiter1974} for history), it is very common to study the local statistics of sequences. A key challenge is to show that the correlations of such sequences converge to the Poissonian limit. Since
    higher $\ell$-point correlation 
    functions tend not to concentrate 
    in an $L^2$-sense, very little is known 
    about higher-order correlation functions.
    This was first
    observed by Rudnick 
    and Sarnak \cite{RS};
    see also \cite[Appendix C]{TechnauWalker2020} for more discussion. 
    It has long been expected (see for example \cite{Marklof2000,R,RS,RSZ,RS2024}) that for almost any choice of $\alpha$ the sequence $x(n)$ should have Poissonian local statistics for any $d\ge 2$. 

    For sequences on the torus there have been few deterministic or metric results. Pellegrinotti \cite{P} showed that the first four moments of the distribution of quadratic polynomials with random coefficients converge to the Poissonian limit. Surprisingly, Sinai \cite{S} showed that the moments beyond the $26^{\text{th}}$ do not converge to the Poissonian limit. In fact, this observation led some to believe that in order to produce Poissonian statistics for higher moments, one needs an increasing number of real parameters, Theorem \ref{thm:main corr} contradicts this philosophy. For generalised monomials where the exponent $d$ is a small real (e.g $\le 1/3$) the correlations are known to be Poissonian \cite{LST,LT1,LT2} and when the sequence grows with a lacunary rate then the sequence is known to have Poissonian statistics almost everywhere \cite{RZ}.

    The local statistics of real valued sequences (not on the torus) are comparatively well understood (see for example \cite{EMM2005, Marklof2003, KMW2026} for results concerning $\ell=2$ and \cite{ABR,
    %vanderkam1999,vanderkam1999secondarticle,
    vanderkam2000, RudnickKurlberg1999} for higher order correlations -- though this list is incomplete). However, in our setting the periodicity makes it hard to keep track of neighbours which make the gap statistics very hard to control.

    The following extension of Theorem \ref{thm:main corr} shows that we can replace the Lebesgue measure on $[0,1)$ with any measure satisfying a 
    growth condition on its Fourier coefficients. 
    \begin{theorem}\label{thm Hausdorff correlation}
    Let $d>d_{4\ell}$. 
    Suppose $\mu$ is a finite Borel measure 
    on $[0,1]$ whose
    Fourier transform $\widehat{\mu}(\xi) = \int_0^1 
    e(-\xi \alpha) d\mu(\alpha)$
    decays on average like 
    $\vert \xi \vert^{-1/2}$
    in the following sense.
    If $0<s<1$, then
    \begin{equation}\label{eq series converges}
    \vert \widehat{\mu}(0)\vert^2 + 
    \sum_{u\neq 0} 
    \vert \widehat{\mu}(u)\vert^2 \cdot \vert u\vert^{s-1}<\infty.
    \end{equation}
    Then, 
    $\cX_{\alpha}^d$
    has Poissonian ${\ell}$-point correlations 
    for $\mu$-almost any $\alpha\in [0,1]$.
    \end{theorem}
    
     We say $\alpha$ is diophantine of type $\tau$ if there exists a $c_\alpha>0$ such that $|\alpha-p/q| \ge c_\alpha q^{-\tau}$ for all coprime $p$ and $q$. It is immediate that the diophantine properties of $\alpha$ will dictate the behavior of the correlations (see for example \cite[Theorem 2]{RSZ}). 
     Rudnick, Sarnak and Zaharescu
     \cite[p. 38]{RSZ} conjecture
     that any diophantine number with sufficiently 
     square-free denominators in 
     its continued fraction convergents 
     should lead to Poissonian behavior of 
     the gaps in $\cX_{\alpha}^2$. An application of Theorem \ref{thm Hausdorff correlation} shows that for almost every diophantine $\alpha$, the sequence $\cX_{\alpha}^d$ will have Poissonian $\ell$-correlations for $d>d_{\ell}$.

    In fact, while Theorem \ref{thm:main corr} applies to a much larger set of $\alpha$, Rudnick and Sarnak \cite[Abstract]{RS} conjecture that, if $\alpha$ is badly approximable and $d\ge 2$ is a fixed integer, then the normalised spacings between elements of $\cX_\alpha^d$ have Poissonian statistics. While there have been some steps towards this conjecture (see e.g. Heath-Brown \cite{Heath-Brown2010}), no one has confirmed it for any degree or statistic. The following corollary states the existence of a `large' subset of badly approximable numbers which satisfy this conjecture with respect to the $\ell$-point correlation function.

     \begin{corollary}\label{cor Hausdorff}
        Fix $\ell>1$ and let $d > d_{4\ell}$. Then there exists a Hausdorff dimension $1$ subset $\cA$ of the badly approximable numbers such that $\cX_{\alpha}^d$ has Poissonian $\ell$-point correlations for all $\alpha \in \cA$.
     \end{corollary}
    
    \begin{proof}
    The $s$-energy
    $$
    I_s(\mu)=\iint_{[0,1]^2} |x-y|^{-s} d \mu(x) d \mu(y)
    $$
    of $\mu$ is well-known to be 
    comparable to its 
    discretisation 
    \eqref{eq series converges},
    see \cite[Corollary 2.7]{hareMohantyRoginskaya2007}.
    Hence,
    Theorem \ref{thm Hausdorff correlation}
    and Frostman's Lemma implies directly 
    that any set $\cS\subseteq[0,1]$ 
    of Hausdorff dimension $1$ contains a 
    set $\cA\subseteq \cS$ 
    so that $\cX_{\alpha}^d$ has Poissonian 
    $\ell$-point correlations for any $\alpha \in \cA$
    once $d>d_{4\ell}$. Choosing $\cS$ to be the set of badly approximable numbers proves the corollary.

    \end{proof}

    \subsection{Stratified point counting}

    Fix a degree $d\ge 2$ and dimension $n$, and consider the homogeneous diagonal form\footnote{In keeping with convention, when discussing the counting problem in projective space we label the points $0$ to $n$.}
    \begin{align}\label{def Q}
        Q_{\vect{a}}(\vect{x}):=  a_0 x_0^d + \dots + a_n x_n^d,
    \end{align}
    with coefficient vector $\vect{a}\in (\Z_{\neq 0})^{n+1}$.
    A natural question in diophantine analysis is to count solutions to the equation $Q_{\vect{a}}(\vect{x})=0$ with entries taken from a box $[-B,B]^{n+1}$. To that end, let $X_n \subset \mathbb{P}^{n}$
    denote the diagonal hypersurface defined by $Q_{\vect{a}}(\vect{x}) = 0 $. For any point $x \in X_n(\mathbb{Q})$, we take a representative $\mathbf{x} =(x_0,\ldots, x_n) \in \mathbb{Z}^{n+1}$ with $\gcd(x_0, \ldots, x_n)=1$. Then we write $
    H(x) = \vert \vert \mathbf{x} \vert \vert_\infty
    $
    and define the counting function
    \begin{align*}
        \mathcal{N}_{Q_{\vect{a}}}(B) : = \# \{ x \in X_n(\mathbb{Q}): H(x) \leq B\}.
    \end{align*}
    For general non-singular hypersurfaces,
    Browning--Heath-Brown~\cite{BHB} showed that 
    \begin{align}
        \mathcal{N}_{Q_{\vect{a}}}(B)\ll_{d,n,\vep}B^{n-1+\vep}\quad \text{ for all } \vep >0, 
    \end{align}
    uniformly in $\vect{a}$. 
    However, one expects that the diagonal structure of the form should allow for improvements. Indeed,
    Salberger--Wooley \cite{SW} showed that once $d$ is large, diagonal solutions dominate. In particular, once $d>(2n+2)^{4n+4}$, this provides the bound
    \begin{align}\label{SW}
        \mathcal{N}_{Q_{\vect{a}}}(B) \ll_{d,n}B^{\lfloor\frac{n+1}{2}\rfloor}. 
    \end{align}
    This bound comes from the diagonal solutions (those lying on maximal dimensional $\mathbb{Q}$-linear subspaces, should such exist) and they show that the off-diagonal solutions contribute $O_{d,n}(B^{\lfloor\frac{n+1}{2}\rfloor-\frac{1}{3}})$.
    
    The purpose of this paper is to improve this upper bound on $\mathcal{N}_{Q_{\vect{a}}}(B)$ by showing that the count depends on the coefficients $\mathbf{a}$ and can be massively reduced for generic\footnote{In the sense of being drawn uniformly
    at random from a homogeneous expanding box.} $\mathbf{a}$. Before stating the main counting theorem, let $m=m(\mathbf{a})$ denote the largest $0 \leq k \leq \frac{n+1}{2}$ such that there exist $k$ disjoint pairs of distinct indices $(i_1,j_1), \ldots, (i_k,j_k)$ with
    \[
    -a_{i_\ell}a_{j_\ell} \in \mathbb{Z}^{[d]} \quad{} \forall \ell=1,\ldots,k
    \]
    where $\mathbb{Z}^{[d]}$ denotes the set of $d^\text{th}$ powers. Furthermore, let
    \begin{align}\label{def L dn}
        L(d,n) : = 1+ \frac{2(d-n+1)}{n^2-n}
    \end{align}
    and let 
    \begin{align*}
        \phi(m,d,n) : = \max_{0\le s \le m} \left(s+  \sum_{r=1}^{n-2s-1}\frac{r+1}{\sqrt[r]{L(d,n)}}\right).
    \end{align*}
    Then our counting theorem states that the count is dominated not just by diagonal terms, but all the ``quasi-diagonal" terms enumerated by $m$.
    \begin{theorem}\label{thm:main-counting}
        Let $d>n> 3$. Let $W$ denote the subset of points $x \in X_n(\mathbb{Q})$ for which there exists a pair of distinct indices $i,j$ such that
        \[
        a_ix_i^d+ a_jx_j^d =0.
        \]
        Then, for all $\vep >0$
        \[
        \#\{ x \in X_n(\mathbb{Q}) \setminus W: H(x) \leq B\}
        \ll B^{\sum_{r=1}^{n-1} \frac{r+1}{{L(d,n)}^{1/r}}+\vep}.
        \]
        Moreover, for the total number of points we have the bound
        \begin{align}\label{N bound}
            \mathcal{N}_{Q_{\vect{a}}}(B) \ll\ B^{\phi(m,d,n) +\vep} .
        \end{align}
        In both bounds, the implied constant is uniform in $\vect{a}$.
    \end{theorem}

\begin{remark}
    In particular, observe that for a fixed $n$, once $d > d_n$ we have that \[
     \#\{ x \in X_n(\mathbb{Q}) \setminus W: H(x) \leq B\} = o(B),
    \] demonstrating that the quasi-diagonal solutions dominate the count in a very strong manner. Similarly, another immediate consequence is that for any $\vep>0$ there exists a $d(\vep)$ such that $d> d(\vep)$ yields
    \[
    \mathcal{N}_{Q_{\vect{a}}}(B) \ll_\vep B^{m+\vep},
    \] where $m$ is defined as above.
\end{remark}

    The condition that $-a_ia_j \in \mathbb{Z}^{[d]}$ characterises when quasi-diagonal solutions coming from linear spaces within the hypersurface arise and, in particular, $m$ is the dimension of the largest $\mathbb{Q}$-linear space within the hypersurface. Note that the exponent $\phi$ varies significantly with $m$. In the generic case one expects $m=0$, hence the exponent is dominated by the $r$ sum, which decays as $d$ gets large and $n$ is fixed. Conversely, in the worst case scenario where the hypersurface contains $\frac{n-1}{2}$ dimensional planes, the first term in the exponent dominates and we recover the asymptotic bound from Salberger--Wooley \eqref{SW}. 
    
    In the circle method literature, these kinds of paucity results have been established for many diagonal systems of equations (see for instance \cite{paucityvmvt,paucitytriples,paucitypairs,paucityskinner}). 
    Two particularly striking examples are the work of Vaughan--Wooley~\cite{MR1438823} on Vinogradov systems and Wooley~\cite{wooleybrudernrobert} on the Br{\"u}dern--Robert system, where the authors produce bounds of the form $O_\vep(B^{\sqrt{2n+3}+\vep})$ and $O_\vep(B^{\sqrt{4n+5}-1+\vep})$, respectively, for the off-diagonal solutions. These go far beyond the square-root cancellation of the Salberger--Wooley result \eqref{SW} and show that the number of solutions is dominated by the asymptotic formula for the diagonal solutions of which one can elaborate many lower order terms before the off-diagonal solutions make any contribution. These results are very much in the same spirit as those of Theorem \ref{thm:main-counting} where the size of the quasi-diagonal solutions dominates, and for sufficiently large $d$ one can see that the lower order terms from these solutions (should one write out the asymptotic, of which we have no need for our later purposes) would also dominate the error from the non quasi-diagonal terms.

    One distinction in our approach is that it works in the regime where the degree is much larger than the number of variables. Another is that our bounds allow for any choice of coefficients whereas most previous works handle specific equations with coefficients in $\{\pm1\}$. These distinctions we have in common with the work of Salberger--Wooley and stem from the principal tool being the determinant method. 
    \begin{remark}
        The determinant method was first developed by Bombieri--Pila~\cite{BP} to show that there were few integral points on real curves, reflecting the fact that dilates of algebraic curves of high genus should have few points. One of the most important features of the method is that the bound produced is independent of the coefficients of the curve. The method was then adapted by Heath-Brown~\cite{HBannals} who developed a $p$-adic analogue and by Salberger~\cite{PLMS} who developed a global method. Building on the perspective of curves of high genus, one could view our results as a quantitative version, in the diagonal case, of the expectation that the rational points on hypersurfaces of general type should be restricted to very few subvarieties (a very weak formulation of the Bombieri--Lang conjecture).
    \end{remark}
    As demonstrated by the remark following Theorem \ref{thm:main-counting}, we are able to go far beyond square-root cancellation when avoiding the quasi-diagonal solutions in the spirit of the strongest known paucity results. In fact, by taking $d$ large enough, we can get arbitrarily strong bounds for non quasi-diagonal solutions. The major input that allows us to control the rational points on diagonal hypersurfaces outside of linear subspaces in such a strong manner is the fact that such points cannot lie on subvarieties of very small degree (c.f.\ Lemma \ref{lem:degreecontrol}).

\subsection{\label{subsec:Application} Nearest neighbor spacing
distribution}

Correlation statistics of a sequence allow us to obtain information about the gap distribution. To define this, we relabel the points $\{x_n\}_{n=1}^{N}$ so that the labels correspond to the position on the torus
$u_1^{(N)} \le u_2^{(N)} \le \cdots \le u_N^{(N)}$.
We define the cumulative gap distribution to be the limit (if it exists) 
\begin{align*}
    P(\cX_\alpha^d)(s) = \lim_{N\to \infty} P_N(\cX_{\alpha}^d)(s) : =  \frac{\#\{ i \le N \ : \ u_{i+1}^{(N)} - u_i^{(N)} < s/N\}}{N}.
\end{align*}
The scaling $s/N$ ensures that this is measuring the gaps on the level of the mean spacing. From Theorem \ref{thm:main corr} one can derive the following corollary which states that the gap distribution of $\cX_{\alpha}^d$ can be bounded by Taylor approximations of $1-e^{-x}$. Thus, the gap distribution is approximately Poissonian. 

    \begin{corollary}\label{cor gap}
    Let $K\geq 1$, and  $d>d_{2K}$ be integers. 
    For almost all $\alpha$, 
    we have 
\begin{equation}\label{eq gap distr convergence}
\sum_{1\le k\leq2K}(-1)^{k+1}\frac{s^{k}}{k!} \leq\liminf_{N\to\infty}  P_N(\cX_{\alpha}^d)(s)  \le\limsup_{N\to\infty}P_N(\cX_{\alpha}^d)(s) \leq\sum_{1\le k\leq2K-1}(-1)^{k+1}\frac{s^{k}}{k!}
\end{equation}
for all $s>0$. Consequently, 
$
\lim_{d\rightarrow \infty}
P(\cX_{\alpha}^d)(s) = 1-e^{-s}
$ for almost all $\alpha$.
\end{corollary}

\begin{proof}
    \eqref{eq gap distr convergence}
    is a direct consequence of
    Theorem \ref{thm:main corr} 
    and known relations between the gap distribution and correlation functions, see 
    \cite[Lemma 11 and (A.2)]{RudnickKurlberg1999}
    and compare
    \cite[Corollary 1.6]{TechnauYesha2020}.
\end{proof}

    \subsection{Further directions}

    The main novelty in this paper is the application of the determinant method to the counting problem arising in the correlations mod $1$ context and the stratified counting which allows us to control averages far more effectively. We have made some effort to tighten the screws in this argument but there are still some sources of loss which we anticipate could be improved. This would likely improve the range of possible degrees in Theorem \ref{thm:main corr}.

    Counting rational points on diagonal hypersurfaces appears in a variety of contexts. For instance, in the representation of a number as a sum of $d^{th}$ powers \cite{HB}. However, slightly less evidently, one can also use estimates like Theorem \ref{thm:main-counting} to count matrices of a given form whose entries come from some arithmetically interesting sets. For example, in \cite{BL} Blomer--Li use an estimate for the number of matrices with a given rank and entries of the form $x^d$ for $x<N$ to show that the $\ell$-point correlations between points on the real line given by diagonal equations in $k$ variables with power $d$ are almost surely Poissonian. The estimate for the number of matrices with fixed rank was then improved upon and generalized to more general arithmetically defined entries in recent work of Mohammadi--Ostafe--Shparlinski \cite{MOS}. Theorem \ref{thm:main-counting} could lead to some improvement in the ranges of both these results. However, na\"ively plugging in our bound does not yield a significant improvement since in both cases the bottleneck in the argument is elsewhere.

    In addition, these kinds of stratified counting theorems could be used in a variety of arithmetic contexts. Indeed, there are the obvious analogues of the various representation problems considered in Marmon~\cite{Marmon} and potential applications to lower order terms in the mean values of Weyl sums. Moreover, all our counting results should be extendable to global fields via the work of Paredes--Sasyk~\cite{globalfields}\footnote{We thank Matteo Verzobio for this observation.}.
    As far as future efforts go, it would be particularly interesting to see results of this kind for non-diagonal hypersurfaces.

    \subsection*{Notation}

    Throughout this paper we use typical Bachmann--Landau notation: for functions $f,g : X \to \R$ defined on some set $X$, we write $f\ll g$ (or $f=O(g)$) to denote that there exists a constant $C>0$ such that $|f(x)|< C |g(x)|$ for all $x\in X$. We write $f \asymp g$ if $f\ll g$ and $g\ll f$ and let $f=o(g)$ 
    mean that $\frac{f(x)}{g(x)} \to 0$. 
    Let $[N]:=\{1, \dots, N\}.$
    Given a Schwartz function $f: \R^m \to \R$, let $\wh{f}$ denotes the Fourier transform:
    \begin{align*}
      \wh{f}(\vect{k}) : = \int_{\R^m} f(\vect{x}) e(-\vect{x}\cdot \vect{k}) \mathrm{d}\vect{x}, \qquad \text{for } \vect{k} \in \R^m.
    \end{align*}
    Here, and throughout we let $e(x):= e^{2\pi i x}$. Finally, we use $\vep>0$ as a small constant which can change from line to line. For brevity of notation, we use the convention $\sqrt[1]{n}=n$. 

    \textbf{Plan of the paper:} We prove the counting theorem, Theorem \ref{thm:main-counting}, in Section \ref{s:counting}. Then we prove the Lebesgue almost everywhere statement, Theorem \ref{thm:main corr}, in Section \ref{s:correlations}. Then we prove the extension to fractal measures, Theorem \ref{thm Hausdorff correlation}, in Section \ref{s:fractal}. We have taken care to use notation that is coherent and remains close enough to the literature to avoid committing any sins. However, since the different sections are rather different in flavor, we will not import the notation between the different sections.

    \section{Proof of Theorem \ref{thm:main-counting}}
    \label{s:counting} 

    Throughout this section, suppose that $n \geq 3$. Fix $a_0, \ldots, a_n \in \Z_{\neq 0}$ and let $X_n \subset \mathbb{P}_{\mathbb{Q}}^n$ denote the projective hypersurface defined by the diagonal equation $Q_{\vect a}(\vect{x}) = a_0x_0^d + \ldots+a_n x_n^d =0$. A \emph{quasi-diagonal} subvariety on $X_n$ are those subsets of the rational points on $X_n$ satisfying a system of equations of the form
    \[
        \sum_{j\in P_1}a_jx_j^d = 0 =\sum_{j\in P_2}a_jx_j^d, 
    \]
    for some partition $P_1 \cup P_2 =\{0,1, \ldots, n\}$ into parts of size at least 2. The lines contained within such subvarieties are called the \emph{standard lines} and are the only lines on $X_n$ (c.f.\ \cite[Ex. 2.5.3]{Debarre}). 

    The following result of Salberger controls the degree of curves on $X_n$ and is the major geometric engine behind our results.
    \begin{theorem}[{\cite[Thm. 9.1]{PLMS}}]\label{thm:salcurves}
        Let $C$ be a closed integral curve on $X_n$ of degree $\delta$ which does not lie on any other diagonal hypersurface of degree $d$ in $\mathbb{P}^n$. Then
        \[
        \delta \geq 1-\frac{2}{n-1}+\frac{2d+2}{n(n-1)} = L(d,n).
        \]
    \end{theorem}

    Salberger~\cite{PLMS} and Marmon~\cite{Marmon} showed that the only curves of small degree belonging to $X_n$ are the standard lines when $n=3$ and 4, respectively. We start by establishing a slightly weaker variant that holds for all $n$.

    \begin{lemma}\label{le no small degree curves}
      Let $C \subset X_n$ be a closed integral curve of degree less than $L(d,n)$.
    Then $C$ is contained within a quasi-diagonal subvariety.
    \end{lemma}

\begin{proof}
    We proceed to show, by strong induction on $n\ge 3$, that given a diagonal hypersurface $X_n\subset \mathbb{P}^n$, the degree of any such curve $C$, not contained in a quasi-diagonal subvariety, is bounded by $L(d,n)$ uniformly in the coefficients. The base case when $n=3$ corresponds to a result of Salberger~\cite[Thm. 9.4]{PLMS}. Let us then suppose that the conclusion holds for all natural numbers between 3 and $n-1$ and let $C \subset X_n$ be an integral curve of degree less than $L(d,n)$.

    Applying Theorem \ref{thm:salcurves} we see that $C$ must lie on another diagonal hypersurface. Let us call this hypersurface $Y$ and suppose that it is cut out by the vanishing of the form $G(\vect x) = \sum_{i=0}^n b_i x_i^d$. Recall that we denote by $Q_{\vect a}(\vect x)$ the form whose vanishing defines $X_n$. 
    Then there exists a $\Z$-linear combination of 
    $Q_{\vect a}$ and $G$ such that at least one 
    of the coefficients is 0. 
    Hence without loss of generality, we may assume that $b_n=0$.
    Now we project 
    $\mathbb{P}^n \rightarrow \mathbb{P}^{n-1}$ onto the first $n$ coordinates. Let $Z$ denote the image of $Y$ under this projection, and $D$ the image of $C$. Note that $Z$ is again a diagonal hypersurface, now inside $\mathbb{P}^{n-1}$. The image $D$ must be either a point or an irreducible curve. However, if it were a point then $C$ would be the line passing through the point $[0:\ldots:0:1]$ which cannot be since $a_n\neq 0$. Therefore $D$ is a curve and since it has degree smaller than $L(d,n) \leq L(d,n-1)$, we conclude by the inductive hypothesis that there exists a subsum $\sum_{i\in I} b_i x_i^d $ vanishing on $C$.

    We would like to draw a similar conclusion but for a subsum involving the $a_i$ not the $b_i$. We project down now from $\mathbb{P}^n$ to $\mathbb{P}^{ \#I-1}$ onto the coordinates lying in $I$. Let $W$ denote the subvariety given by the subsum found above, namely $\sum_{i \in I} b_i x_i^d=0$. Either the image of $C$ under this projection is a point or a curve. If it is a curve then as before we may find another diagonal form in $\mathbb{P}^{\#I -1}$ which also vanishes on the image of $C$. By the exact argument just used, there is a subset $I' \subset I$ and a form $\sum_{i \in I'} c_i x_i^d$ vanishing on the image of $C$. We can therefore project again onto $\mathbb{P}^{\#I'-1}$.

    We repeat this process, reducing the dimension of the ambient projective space by projection, until one of two things happens. Either our curve gets projected onto a point or we eventually project onto a projective space of dimension $\leq 4$. In the latter case, we know by the work of Salberger and Marmon, that the image of $C$ must be a line. In particular, this image must be a line on the projection of $X_n$, which remains a Fermat hypersurface. Since all lines are standard lines, we conclude that there must be a vanishing subsum.

    Otherwise, this procedure has produced an integer $r$ and a projection $\mathbb{P}^n \xrightarrow{\pi} \mathbb{P}^r$ such that the image of $C$ under this projection is a point. Let us suppose, without loss of generality, that point is $[y_0:\ldots:y_{r-1}:1]$. The fiber under $\pi$ above this point is a linear space $L$ inside $\mathbb{P}^n$. Points on the intersection $L \cap X_n$ must satisfy an equation of the shape
    \begin{equation}\label{eq:neweq}
    cx_r^d + a_{r+1}x_{r+1}^d + \ldots + a_nx_n^d =0,
    \end{equation}
    where
    \[
    c = a_0y_0^d + \ldots a_{r-1}y_{r-1}^d + a_r
    \]
    If $c=0$ then we are done. Otherwise, we may apply the inductive hypothesis to \eqref{eq:neweq}. This will produce a subset $I$ of indices whose subsum vanishes. If $r \not \in I$ then we have found a vanishing subsum of the original equation. Otherwise the sum of the indices not in $I$ must also sum to 0 and thus we have the claim.
\end{proof}

We now extend this conclusion to subvarieties of greater dimension.

\begin{lemma}\label{lem:degreecontrol}
    Let $Y \subset X_n$ be a subvariety of positive codimension and of degree less than $L(d,n)$. Then there exists a subset $I \subset \{0,\ldots,  n\}$ such that $\#I \geq 2$ and  such that the form $\sum_{i \in I} a_i x_i^d$ vanishes on $Y$.
\end{lemma}

\begin{proof}
     The proof follows that of \cite[Lemma 4.2]{SW}, however, in place of Lemma 4.1 therein we may apply our Lemma \ref{le no small degree curves}.
\end{proof}

With this result in hand we may now apply the determinant method. 
\begin{lemma}\label{le number of aux surfaces}
There is a collection of 
$$
O_{d,n,\vep}\left(B^{\sum_{r=3}^{n-1}\frac{r+1}{\sqrt[r]{L(d,n)}}+\vep}\right)
$$ integral surfaces of degree $O_{d,n,\vep}(1)$ such that every point of $X_n$ of height at most $B$ lies in one of these subvarieties. 
\end{lemma}

Our result is based on \cite[Lemma 3.4]{SW}, however in that result the authors only cover the points of bounded height with subvarieties of codimension $(n-2)/2$. We are able to go deeper because Lemma \ref{lem:degreecontrol} affords us control over the degree of \emph{all} subvarieties which avoid vanishing subsums. 

\begin{proof}
    This proof is exactly like that of \cite[Lemma 3.4]{SW}, inductively applying \cite[Lemma 3.3]{SW}.
\end{proof}

Now that the points have been restricted to surfaces, we may apply the following result of Salberger. Note that the below is not how the result is stated in \cite{PLMS} but the same proof provides the subsequent result; see also Verzobio 
\cite[Lemma 2.2]{verzobio2025}
for a detailed discussion.
\begin{theorem}[{\cite[Thm. 6.1]{PLMS}}]\label{lem:matteo}
    Let $\vep >0$ and let $S \subset \mathbb{P}_\mathbb{Q}^n$ be a geometrically integral surface of degree $d$. Then the number of rational points on $S$ of height at most $B$ which do not lie on any irreducible curves of degree at most $e-1$ is bounded by
    \[
    O_{n,d,e,\vep}\left(
    B^{\frac{3}{\sqrt{d}}+\vep} + B^{\frac{3}{2\sqrt{d}}+\frac{2}{e}+\vep}\right).
    \]
\end{theorem}

As a consequence, we are able to count points which do not satisfy $\sum_{i \in I} a_i x_i^d =0$ for any subset $I \subset\{0, \ldots, n\}$ for which $2 \leq \#I \leq n-1$. We say that for such points, there is \emph{no vanishing subsum}.
\begin{corollary}\label{cor:nosubsum}
    The number of points of height at most $B$ on $X_n$ for which there does not exist a vanishing subsum is \begin{align}\label{bjeb}
    O_{d,n,\vep}\left(B^{\sum_{r=3}^{n-1}\frac{r+1}{\sqrt[r]{L(d,n)}}}\left(B^{\frac{3}{\sqrt{L(d,n)}}+\vep}+B^{\frac{3}{2 \sqrt{L(d,n)}}+ \frac 2{L(d,n)} + \vep}\right)\right).
    \end{align}
In particular, this bound is independent of the choice of coefficients $\vect{a}$. Moreover, \eqref{bjeb} can be further bounded by
\begin{align}\label{bjeby}
    O_{d,n,\vep}\left(B^{\sum_{r=1}^{n-1}\frac{r+1}{\sqrt[r]{L(d,n)}} +\vep}\right).
    \end{align}
\end{corollary}

This affirms the first claim of Theorem \ref{thm:main-counting}.
\begin{proof}
    By  Lemma \ref{le number of aux surfaces}, the rational points lie on at most 
    $O_{d,n,\vep}\left(B^{\sum_{r=3}^{n-1}\frac{r+1}{\sqrt[r]{L(d,n)}}+\vep}\right)
    $ surfaces. 
    These surfaces are integral but not necessarily geometrically integral, however we may reduce to the case of geometrically integral surfaces by following the argument of \cite[Thm 2.1]{Salberger2}. 
    By Lemma \ref{lem:degreecontrol}, we know that the points with no vanishing subsums must lie on surfaces of degree at least $L(d,n)$. By Lemma 
    \ref{le no small degree curves}, any curve on such a surface also has degree at least $L(d,n)$. Therefore, we conclude by applying Theorem \ref{lem:matteo}.
\end{proof}

Finally, we record some results on counting zeros of forms in a small number of variables.
\begin{lemma}[{\cite[Thm. 5]{HBannals}}]\label{lem:curves}
    Let $C$ be an irreducible curve in $\mathbb{P}^3$ of degree $d$. Then the number of points of height at most $B$ on $C$ is $O_{d,\vep}(B^{\frac{2}{d}+\vep})$.
\end{lemma}

\begin{lemma}\label{le 2 dim count}
    If $a,b$ are non-zero integers and $B\geq 1$,
    then $$
    \#\{\vert x\vert, \vert y \vert \leq B : ax^d = by^d \} \leq 8B+1.
    $$
\end{lemma}
\begin{proof}
    We may assume without loss of generality that $a$ and $b$ are coprime. If $x=0$ then we must have $y=0$ and vice versa, since $ab \neq 0$. Henceforth assume that $x$ and $y$ are non-zero. Write
    \begin{align*}
        x &=x_1x_2\\
        y&=y_1y_2,
    \end{align*}
    where $x_1$ and $y_1$ are defined by the relation $p|x_1$ implies $p|b$ and where $q|y_1$ implies $q|a$ for all prime $p$ and $q$ (or rather $x_1 \mid b^\infty, y_1 \mid a^\infty$) and $\gcd(x_2,b) = \gcd(y_2,a)=1$. We deduce from the coprimality that $b \mid x_1^d$ and $a \mid y_1^d$, thus our equation can be written
    \[
    \frac{x_1^d}{b} x_2^d = \frac{y_1^d}{a}y_2^d.
    \]
    Hence $x_2^d \mid y_2^d$ and vice versa, hence $x_2^d = y_2^d$. Therefore there are at most $4B$ choices for $(x_2,y_2)$. Finally, substituting this equality into the original equation yields 
    \[ax_1^d = by_2^d\] which implies $x_1^d = b$ and $y_1^d = a$ from which we deduce that there are at most 4 choices for $(x_1,y_1)$. Note that $(-x_1, -x_2)$ produces the same $x$ as $(x_1,x_2)$ hence the claimed bound.
\end{proof}

With all this in hand, we can conclude our upper bound for the point count.
\begin{theorem}\label{thm main counting}
Let $d>n>3$.
Let $m$ be the largest number of disjoint pairs of distinct indices $\{i,j\}$ such that $-a_ia_j$ is a $d^\text{th}$ power. Then, as $B \rightarrow \infty$, we have
\[
\mathcal{N}_{Q_{\vect{a}}}(B)
\ll_{\vep, d,n} 
\sum_{s=0}^m B^{s + \sum_{r=1}^{n-2s-1} \frac{r+1}{\sqrt[r]{L(d,n)}}  + \vep}.
\]
\end{theorem}

\begin{proof}
    We proceed via strong induction on $n\ge 4$. When bounding the number of zeros of binary forms we use Lemma \ref{le 2 dim count}, for zeros of ternary forms we use Lemma \ref{lem:curves}, for zeros of quaternary forms we require the following bound, \eqref{tilde bound}. To that end, let $\wt{n}=3$, then the number of points avoiding a vanishing subsum is counted by Theorem \ref{lem:matteo}.
    Taking Theorem \ref{thm:salcurves} into account, 
    we infer that the number of such points is
    \begin{align}\label{tilde bound}
    \cN_{Q_{\wt{\vect{a}}}}(B)\ll_{\varepsilon,d}
    B^{\frac{3}{\sqrt{d}}+\vep}+B^{\frac{3}{2\sqrt{d}}+ \frac 2{L(d,3)} + \vep},
    \end{align}
    where $Q_{\wt{\vect{a}}}$ is a quaternary diagonal form.
    Should there exist a subsum, it must be a pair of binary subsums. In this case we have $m=2$ and we may apply Lemma \ref{le 2 dim count}.
    Hence the number of points is
    \[
    \begin{cases}
    O_{\varepsilon,d}\left(
    B^{\frac{3}{\sqrt{d}}+\vep}+B^{\frac{3}{2\sqrt{d}}+ \frac 2{L(d,3)} + \vep} \right)  = O_{\varepsilon,d}\left( B^{\frac{3}{\sqrt{L(d,3)}}+\frac{2}{L(d,3)}+\vep}\right) \text { if } m=0\\
    O\left(B^2\right) \text{ if } m=2,
    \end{cases}
    \] completing the proof in this case.

    Suppose that $n=4$. The total number of points avoiding vanishing subsums is provided by Corollary \ref{cor:nosubsum} which yields
    \[
    O_{\varepsilon,d} \left(
    B^{\frac{4}{\sqrt[3]{L(d,4)}}+\vep}\left( B^{\frac{3}{\sqrt{L(d,4)}}} + B^{\frac{3}{2\sqrt{L(d,4)}}+\frac{2}{L(d,4)}+\vep} \right)\right).
    \]
    If a quaternary subsum vanishes then we have one choice for the remaining variable and the bound
    \[
    O_{\varepsilon,d}\left(B^2 + B^{\frac{3}{\sqrt{L(d,4)}}+\frac{2}{L(d,4)}+\vep}\right)
    \] for the four relevant variables, by the previous paragraph.
    If a point is not enumerated by either of these counts then there must be a set of 3 indices $\{i_1,i_2,i_3\} \subset \{0,\ldots,4\}$ such that
    \[
    a_{i_1}x_{i_1}^d +a_{i_2}x_{i_2}^d+a_{i_3}x_{i_3}^d=0.
    \]
    The number of potential solutions to the above display by Lemma \ref{lem:curves} is $O_\varepsilon(B^{\frac{2}{L(d,n)}+\vep})$. By Lemma \ref{le 2 dim count}, there are $O(B)$ choices for the remaining two $x_i$ variables. In total we infer 
    \[
    N_{Q_{\vect{a}}}(B)
    \ll_{\varepsilon,d,n}
    B^2 + B^{1 + \frac{2}{L(d,n)}+\vep} + B^{\frac{4}{\sqrt[3]{L(d,n)}}+\frac{3}{\sqrt{L(d,n)}}+\frac{2}{L(d,n)}+\vep}.
    \]
    
    Now suppose the claim holds for all $3<j<n$. The number of points for which no proper subsum vanishes is counted by \eqref{bjeby} which gives
    \[
    O_{\varepsilon,d,n}\left(
    B^{\sum_{r=1}^{n-1}\frac{r+1}{\sqrt[r]{L(d,n)}} + \vep}
    \right).
    \]
    For any remaining point $\vect x$, there exists a minimal partition $I_1, \ldots,I_j$ of $\{0,1,\ldots, n\}$
    such that each subsum $\sum_{i \in I_k} a_i x_i^d$ vanishes and such that there is no partition satisfying the same property with a greater number of parts in the partition.  If $I_k$ is a singleton then there is only 1 solution. For any $I_k$ of size 2, we apply the Lemma\ref{le 2 dim count}. For any $I_k$ of size 3, we apply Lemma \ref{lem:curves}. Any other part in the partition has size between 4 and $n$, thus we may apply the inductive hypothesis.

    In particular, the number of points which have exactly $s$ binary vanishing subsums is bounded by $B^s$ times the number of ways that the remaining $n+1-2s$ variables can vanish without there being smaller vanishing subsum. This provides precisely the $s^\text{th}$ term in the sum.
 \end{proof}

\begin{remark}
    The conclusion of Theorem \ref{thm main counting} is probably suboptimal as currently written. This is because at each step when one removes variables corresponding to subsums, one should be able to conclude that the degree bound on subvarieties of the subsums is $L(d, \dim (\text{subsum}))$ rather than the full $L(d,n)$.
\end{remark}

    \section{Proof of Theorem \ref{thm:main corr}}
    \label{s:correlations}
    With Theorem \ref{thm:main-counting} is at hand, the proof of Theorem \ref{thm:main corr} follows a standard procedure, first applied in this context by Rudnick and Sarnak \cite{RS} but later improved for pair correlations by Aistleitner, Larcher, Lewko \cite{AistleitnerLarcherLewko2017} and generalised to higher correlations by Technau and Yesha \cite{TechnauYesha2020}. Namely, \cite[Proposition 7.1]{TechnauYesha2020} implies that Theorem \ref{thm:main corr} follows from the following bounds via an approximation argument.  Let $\cF_{\ell}\subset C_c^{\infty}(\R^{\ell-1})$ denote the set of rectangular smooth functions, i.e.
    \begin{align*}
        \cF_{\ell} : = \bigg\{f(\vect{x}) = \prod_{i=1}^{\ell-1} f_i(x_i), \quad \text{ where } f_1,\ldots,f_{\ell-1} \in C_c^{\infty}(\R^{\ell-1}) \bigg\}.
    \end{align*}

    \begin{lemma}
        Let $\ell\geq 2$.
        For $\cX_\alpha^d$ as defined in 
        \eqref{def ell corr func}, $d > d_\ell$ and $f\in \cF_{\ell}$
        there exists an $\vep>0$ such that
        \begin{align}\label{eq expectation}
            \int_{0}^{1} R_{\ell}^N(\cX_{\alpha}^d, f) \, d\alpha
            = \expect{f} + O_{f,\varepsilon}(N^{-\vep})
        \end{align}
        and
        \begin{align}\label{eq variance}
            \int_{0}^{1} (R_{\ell}^N(\cX_{\alpha}^d,f))^2 
            \, d\alpha
            = 
            \Bigg( \int_{0}^{1} R_{\ell}^N(\cX_{\alpha}^d,f) 
            \, d\alpha \bigg)^2+ O_{f,\varepsilon}(N^{-\vep}).
        \end{align}
    \end{lemma}

    \begin{proof}

        First, apply Poisson summation and a change of variables to the sum in $\vect{k}$ in the definition of the $\ell$-correlation, that is 
        \begin{align*}
                R_{\ell}^N(\cX_\alpha^d,f)  &=
            \frac{1}{N} 
            \sum_{\vect{n}\in [N]^\ell}^{\ast}
            \sum_{\vect{k}\in\Z^{\ell-1}} 
            f  (N(x(n_1)-x(n_2) +k_1, \dots , x(n_{\ell-1})-    x(n_\ell)+k_{\ell-1})  )    ,\\
            &= 
            \frac{1}{N^{\ell}}\sum_{\vect{n}\in [N]^\ell}^{\ast}
            \sum_{\vect{a}\in\cA_{\ell}(N)}   \wh{f} \bigg(\frac{\vect{a}}{N}\bigg) e(\alpha (\vect{a}\cdot\vect{n}^d)) + O(N^{-\vep}),
        \end{align*}
        where 
        \begin{equation}\label{def A_ell}
            \cA_{\ell}(N):= \{(\vect{a},a_\ell) 
           \ : \vect{a}\in(
            \Z^{\ell-1}\setminus\{\vect{0}\})
            \cap[-N^{1+\varepsilon},
            N^{1+\varepsilon}]^{\ell-1}, \,
            a_\ell = -a_1-\dots - a_{\ell-1} \}.
        \end{equation}
        Note that the fast decay of Fourier coefficients allows us 
        to keep $|a_i| \le N^{1+\vep}$ for all $i\leq \ell$.
        Hence, 
        \begin{equation}\label{eq Fourier 
        expansion}
            R_{\ell}^N(\cX_\alpha^d,f)= \expect{f} + 
            \frac{1}{N^{\ell}}\sum_{\vect{n}\in [N]^\ell}^{\ast}
            \sum_{\vect{a}\in\cA_\ell(N)}   \wh{f} \bigg(\frac{\vect{a}}{N}\bigg) e(\alpha Q_{\vect{a}}(\vect{n})) + O_{f,\varepsilon}(N^{-\vep})
            ,
        \end{equation}
        where $\vect{n}^d := (n_1^d, \dots, n_\ell^d)$ and $Q_{\vect{a}}(\vect{n}) : = \vect{a}\cdot \vect{n}^d$.

       Henceforth, denote 
        \[
        \cE := \frac{1}{N^{\ell}}\sum_{\vect{n}\in [N]^\ell}^{\ast}
            \sum_{\vect{a}\in\cA_\ell(N)}   \wh{f} \bigg(\frac{\vect{a}}{N}\bigg) e(\alpha Q_{\vect{a}}(\vect{n})).
        \]
        Note that the integral is $0$ unless $Q_{\vect{a}}(\vect{n})$ is $0$. Hence, on fixing a small $\vep>0$, we have
        \begin{align*}
            \cE \ll_f \frac{1}{N^\ell} \sum_{\vect{a}\in \cA_\ell(N)}
            \#\{ \vect{n} \in [N]^\ell \text{ distinct}: \ Q_{\vect{a}}(\vect{n})=0 \}.
        \end{align*}
        By an inductive argument  on the different correlation orders we can take $\vert a_i\vert >0$ for all $i\leq \ell$. In fact, this is where the condition that the $x_i$ be distinct appears, if not, the lower order correlations with some of the $a_i$ being $0$ cannot be bounded by $O(N^{-\vep})$ (see \cite[Section 3]{LT1} for a detailed explanation of this argument).
        Given $\vect{a} \in \cA_{\ell}(N)$, let $0\le m(\vect{a})\le \ell/2$ denote the largest number of disjoint distinct pairs $(i,j)$ such that $-a_i a_j \in \mathbb{Z}^{[d]}$. Further, let 
    \begin{align*}
        \cA_{\ell}(N,m) : = \{\vect{a} \in \cA_{\ell}(N): 
        \ m(\vect{a}) = m , 
        \ |a_i| >0\, \,
        \forall_{1\leq i\leq \ell} \}.
    \end{align*}
    Then, by Theorem \ref{thm:main-counting},
    \begin{align*}
        \cE_m:=\sum_{\vect{a} \in \cA_{\ell}(N,m)}  \#\{ \vect{n} \in [N]^\ell \text{ distinct}: 
        Q_{\vect{a}}(\vect{n})=0 \} \ll 
        N^{\phi(m,d,\ell)+\vep} \# 
        \cA_{\ell}(N,m).
    \end{align*}
    By a result of Tolev~\cite{tolev} (see also \cite[Thm. 2.2]{MR4345823}) the number of ways to choose $a_1,a_2 \leq N^{1+\vep}$ so that $-a_1a_2 \in \mathbb{Z}^{[d]}$ is $O(N^{\frac{2}{d}+\vep} (\log N)^{d-1})$ hence we save a factor of $N^{2- \frac{2}{d}}$. Moreover, recall that $a_\ell$ is determined once the first 
    $\ell-1$ components of $\vect{a}$ are fixed. Hence,
    \begin{align*}
        |\cA_{\ell}(N,m)| \ll 
        \begin{cases}
         N^{\ell-1 - m(2-2/d)+\vep} & \text{ if } 0\le m < \ell/2        \\
        N^{\ell-1/d+\vep}  & \text{ if } m=\ell/2.  
        \end{cases}
    \end{align*}
    Thus
    \begin{align*}
        \cE &\ll \frac{1}{N^{\ell}}\sum_{m=0}^{\ell/2} \cE_m   \\
            &\ll \frac{N^{\vep}}{N^{\ell}} \sum_{m=0}^{\ell/2} N^{\phi(m,d,\ell)} |\cA_{\ell}(N,m)|\\
            &\ll\frac{N^{\vep}}{N^{\ell}}\bigg( N^{\phi(0,d,\ell) + \ell-1} + N^{\ell-1/d} + \sum_{m=1}^{\lfloor\ell/2-1 \rfloor} N^{\phi(m,d,\ell)} N^{\ell-m(2-2/d)-1} \bigg).
    \end{align*}
    Note that, for $m$ fixed
    \begin{align*}
        \phi(m,d,\ell) + \ell - m(2-2/d) -1  &\le \max_{0\le s \le m} \left(s-m  \right) +\ell-1 +\sum_{r=1}^{\ell-1} \frac{r+1}{\sqrt[r]{L(d,n)}}\\
        &\le \ell-1 +\sum_{r=1}^{\ell-1} \frac{r+1}{\sqrt[r]{L(d,n)}} = \phi(0,d,\ell)+\ell-1.
    \end{align*}
    Hence,
    \begin{align*}
        \cE \ll N^{\phi(0,d,\ell)-1} + N^{-1/d+\vep}.
    \end{align*}
    We conclude that the expected value converges whenever 
    \begin{align}\label{eq phi at most one}
        \phi(0,d,\ell) < 1.
    \end{align}

    Turning now to the variance, the proof is very similar. Using \eqref{eq expectation}, 
    the second moment equals
    \begin{align*}
        \int_{0}^{1} (R_{\ell}^N(\cX_{\alpha}^d,f))^2 
            \, d\alpha
        &=
        \frac{1}{N^{2\ell}} \int_0^1\left|\sum_{\vect{n}\in [N]^\ell}^{\ast} 
            \sum_{\vect{a}\in\Z^{\ell-1}}   \wh{f} \bigg(\frac{\vect{a}}{N}\bigg) e(\alpha (\vect{a}\cdot\vect{n}^d))\right|^2\rd \alpha\\
            &= \expect{f}^2 + \frac{1}{N^{2\ell}} \int_0^1\left|\sum_{\vect{n}\in [N]^\ell}^{\ast}
            \sum_{\vect{a}\in\cA_{\ell}}   \wh{f} \bigg(\frac{\vect{a}}{N}\bigg) e(\alpha (\vect{a}\cdot\vect{n}^d))\right|^2\rd \alpha + O_{f,\varepsilon}(N^{-\vep}).
    \end{align*}
    Expanding the square
    and computing 
    the integral in $\alpha$ leaves us to bound 
    \begin{align}\label{var E}
        \cE : = \frac{1}{N^{2\ell}} \sum_{\vect{a}\in \cA_\ell} \sum_{\vect{b}\in \cA_\ell} \#\{\vect{m},\vect{n}\in [N]^{\ell,\ast} \ : \ Q(\vect{a},\vect{b},\vect{m},\vect{n}) = 0 \},
    \end{align}
    where $[N]^{\ell,\ast}$ is the set of $[N]^{\ell}$ with distinct entries and $Q(\vect{a},\vect{b},\vect{m},\vect{n}) = Q(\vect{a},\vect{m}) - Q(\vect{b},\vect{n})$.

    If we apply the same estimates
    leading up to \eqref{eq phi at most one},
    we arrive at
    \begin{align*}
        \cE \ll N^{-2/d+\vep} + N^{\phi(0,d,2\ell)-2+\vep}.
    \end{align*}
    Note that, since $d > d_{\ell}$ we have $\phi(0,d,2\ell)<2$.

    \end{proof}

    \section{Proof of Theorem \ref{thm Hausdorff correlation}}
    \label{s:fractal}

    The following lemmas allow us to control how much the $\mu$-integration of a trigonometric polynomial
    deviates from the Lebesgue integration $[0,1]$.
    In this section, we use the shorthand notation 
    $L^{p}=L^{p}([0,1])$ for $p\geq 1$.

    \begin{lemma}\label{le Cauchy bound}
     Let 
     $
     T(\xi)=\sum_{\vert u \vert \leq U} 
     c_u e(-\xi u)
     $ be a trigonometric polynomial 
     of degree $U$ where all $c_u\in \C$.
    Let $\mu$ be as in Theorem \ref{thm Hausdorff correlation}.
    Then for all $\rho\in (0,1),$
    \begin{equation}\label{eq energy bound}
    \int_{0}^{1} T(\xi)\, d \mu(\xi) 
    \ll_{\rho,\mu} U^{\rho}\|T\|_{L^2}.
    \end{equation}
    \end{lemma}
    \begin{proof}
Notice, that for $0<s < 1$ we have
$$
\int_{0}^{1} T(\xi)\, d \mu(\xi)=\sum_{|u| \leq U} c_u \widehat{\mu}(u)
=\sum_{|u| \leq U} c_u 
u^{\frac{1-s}{2}} \widehat{\mu}(u) u^{\frac{s-1}{2}}.
$$
By the Cauchy--Schwarz inequality in $\R^{2U+1}$,
$$
\Bigg\vert 
\int_{0}^{1} T(\xi)\, d \mu(\xi) 
\Bigg\vert^2 \leq\left(\sum_{|u| \leq U}
\left|c_u\right|^2 
u^{1-s}\right)
\cdot 
\left(\sum_{|u| \leq U}|
\widehat{\mu}(u)|^2
\vert u\vert^{s-1}\right).
$$
The summation in the first bracket 
is at most $U^{1-s}\|T\|_{L^2}^2$.
By \eqref{eq series converges} 
the summation in the second 
bracket is at most a constant $C_{s,\mu}>0$.
Choosing 
$s=1-2\rho$ and taking roots
completes the proof.
\end{proof}
We now prove a slightly more general
statement than needed
as we believe the transference mechanism
is of independent interest. 

\begin{lemma}\label{le transfer fractal}
Let $C>1$ be fixed.
For each integer $N\geq 1$,
let 
$T_N: [0,1] \rightarrow \C$
be a trigonometric polynomial of 
degree $U_N=O(N^C)$.
Assume $\int_{0}^{1} T_N(\alpha)\, d\alpha =0$ and
suppose there exists an
$\varepsilon>0$ so that
\begin{equation}\label{eq fourth assum}
    \int_0^1
    \vert T_N(\alpha)\vert^4 
    d \alpha \ll_\varepsilon N^{-2\varepsilon}.
\end{equation}
Then, there exists an increasing
sequence of integers $(N_m)_m$ with 
$N_m\sim m^{3/\varepsilon}$ such that
\begin{equation}\label{eq limit on subsequence}
\lim_{m\rightarrow \infty}
T_{N_m}(\alpha) = 0
\end{equation}
for $\mu$-almost all $\alpha\in[0,1]$ where 
$\mu$ is as in 
Theorem \ref{thm Hausdorff correlation}.
\end{lemma}
\begin{proof}
First, we show 
$\int_{0}^{1} T_N(\alpha)\, d\mu(\alpha)=o(1)$. 
Let
$$
T_N(\alpha)=\sum_{ \vert u \vert \leq U_N} 
     c_{u,N} e(-\alpha u).
$$
The Cauchy--Schwarz 
inequality and 
\eqref{eq fourth assum} imply
$
\Vert T_N(\alpha) \Vert_{L^2}
\leq 
\Vert T_N(\alpha)\Vert_{L^4}^{1/2}
\ll_\varepsilon N^{-\varepsilon}
$.
Thus, \eqref{eq energy bound} yields
for any $\rho\in(0,1)$ that

$$
\int_{0}^{1} T_N(\alpha)\, d \mu(\alpha) 
    \ll_{\rho,\mu} 
    N^{C \rho}\|T_N\|_{L^2}
    \ll N^{C \rho} N^{-\varepsilon}.
$$
Upon choosing $\rho\in (0,\varepsilon/C)$,
we deduce $\int_{0}^{1} T_N(\alpha)\, d\mu(\alpha)=o(1)$.
Next, we upper-bound the $\mu$-measure of
$P_N=
\{\alpha \in [0,1]:
\vert T_N(\alpha)   \vert > 
N^{-\frac{\varepsilon}{10}}\}
$.
To this end, we write 
$$
\vert T_N(\alpha)\vert^2= 
\sum_{ \vert u \vert \leq 2U_N} 
     k_{u,N} e(-\alpha u)
     \quad \mathrm{where}\quad 
     k_{u,N}=\sum_{\vert v\vert \leq u}
     c_{v,N}\cdot  \overline{c_{u-v,N}}.
$$
Applying \eqref{eq energy bound} 
with the trigonometric 
polynomial $\vert \widetilde{T}_N(\alpha)\vert^2$
and using \eqref{eq fourth assum}
produces the estimate
$$
\int_{0}^{1} 
\vert T_N(\alpha)\vert^2\, d \mu(\alpha)
\ll_{\rho,\mu} 
(2U)^{\rho} \|T_N\|_{L^4}^2
\ll_\mu N^{C\rho} N^{-\varepsilon}.
$$
Upon choosing $\rho$ small, 
we see that the right 
hand side is 
$O_{\varepsilon}(N^{-\varepsilon/2})$.
Chebyshev's inequality yields 
$\mu(P_N)\ll_{\varepsilon} N^{-\varepsilon/2}$.
The convergence Borel--Cantelli lemma 
shows that the limsup set $P$ of the $P_{N_m}$
is a $\mu$-null set because 
$$
\sum_{m\geq 1} \mu(P_{N_m})\ll 
\sum_{m\geq 1} N_m^{-\varepsilon/2}\ll 
\sum_{m\geq 1} m^{-3/2}
$$
converges. Taking $\cG= S\setminus P$
and noting that \eqref{eq limit on subsequence}
holds for any $\alpha \in \cG$ completes the proof.
\end{proof}
Now we are in a position to prove Theorem \ref{thm Hausdorff correlation}.
\begin{proof}[Proof of Theorem \ref{thm Hausdorff correlation}.]
By the Fourier series expansion 
\eqref{eq Fourier expansion},
we have 
$$
R_{\ell}^N(\cX_\alpha^d, f)= 
\expect{f} +  c_{0,N} + O(N^{-\varepsilon})
            + \sum_{u\neq 0}
            c_{u,N} e(-u\alpha)
\,\,
\mathrm{where}
\,\,
 c_{u,N}:=\frac{1}{N^\ell}\sum_{\vect{n}\in [N]^{\ell}}^{\ast}
        \sum_{\vect{a}\in\Z^{{\ell}-1}\setminus\{\vect{0}\}}
         \wh{f} \bigg(\frac{\vect{a}}{N}\bigg)
        1(Q_{\vect{a}}(\vect{n}))=u).
$$
By \eqref{eq expectation}, we know
$c_{0,N}=o(1)$.
Let $U_N=N^{1+d+\varepsilon}$ and 
$$
T_N= \sum_{1 \leq \vert u \vert \leq U_N}
            c_{u,N} e(-u\alpha).
$$
Then the rapid decay of $\widehat{f}$ implies
$$
R_{{\ell}}^N(\cX_\alpha^d, f) = 
\expect{f} +  c_{0,N} + O(N^{-\varepsilon})
+ T_N + O(N^{-\vep}).
$$
To use Lemma \ref{le transfer fractal},
we bound 
$\int_0^1
    \vert T_N(\alpha)\vert^4 
    d \alpha$.
    
Indeed,
$$
    \int_0^1
    \vert T_N(\alpha) \vert^4 
    d \alpha\ll \frac{1}{N^{4{\ell}}}\sum_{j\leq 4}
            \sum_{\vect{n}_j\in [N]^{\ell}}^{\ast}\,
        \sum_{\substack{
        \vect{a}_j\in\Z^{{\ell}}\setminus\{\vect{0}\}\\ 
        \Vert \vect{a}_j \Vert_\infty \leq  N^{1+\vep}}}
        1\big(Q_{\vect{a}_1}(\vect{n}_1) - Q_{\vect{a}_2}(\vect{n}_2)=
        Q_{\vect{a}_3}(\vect{n}_3) - Q_{\vect{a}_4}(\vect{n}_4)\big).
$$
The counting condition is now relaxed to 
count non-zero $\vect{b}\in \Z^{4{\ell}}$ 
with $b_{4\ell}=b_1+\ldots+b_{4\ell-1}$,
and $\vect{m}\in [N]^{4{\ell}}$ 
so that $Q_{\vect{b}}(\vect{m})=0$
while $\Vert \vect{b}\Vert_\infty \leq 
N^{1+\vep}$, and $\Vert \vect{m} \Vert_\infty \leq 
N$.
Because $d>d_{4{\ell}}$,
we obtain this information from 
\eqref{eq expectation}.
The upshot, by a typical Borel--Cantelli and Chebyshev argument, is that 
for $\mu$-almost any $\alpha\in[0,1]$ we have
$R_{{\ell}}^{N_m}(\cX_\alpha^d,f)
\rightarrow \expect{f}$
as $m\rightarrow\infty$,
for any fixed $(N_m)_m$ 
with $N_m\sim m^{3/\varepsilon}$. 
Upgrading the convergence along a sub-lacunary subsequence
to convergence as $N\rightarrow\infty$
is a routine sandwiching argument, 
see \cite[Lemma 7.2]{TechnauYesha2020}.
\end{proof}

\section*{Acknowledgments} This work was supported by a grant from the Simons Foundation [SFI-MPS-TSM-00013410,CL]. NR is funded by FWF project ESP 441-NBL. NT was supported by the University of Wisconsin–Madison, Office of 
the Vice Chancellor for Research with funding from the Wisconsin Alumni Research Foundation.
We are grateful to Matteo Verzobio for helpful conversations regarding Theorem \ref{thm:main-counting}.
We thank Tim Browning, Jens Marklof, Zeev Rudnick, 
Per Salberger, Peter Sarnak, Igor Shparlinski, and Trevor Wooley for comments on early drafts.

  \bibliographystyle{alpha}
  \bibliography{biblio}

@article{KMW2026,
  title={Values of ternary quadratic forms at integers and the {B}erry--{T}abor conjecture for 3-tori},
  author={Kim, W. and Marklof, J. and Welsh, M.},
  journal={Preprint, arXiv:2601.03209},
}

@article {P,
    AUTHOR = {Pellegrinotti, A.},
     TITLE = {Evidence for the {P}oisson distribution for quasi-energies in
              the quantum kicked-rotator model},
   JOURNAL = {J. Statist. Phys.},
  FJOURNAL = {Journal of Statistical Physics},
    VOLUME = {53},
      YEAR = {1988},
    NUMBER = {5-6},
     PAGES = {1327--1336},
      ISSN = {0022-4715,1572-9613},
   MRCLASS = {28D05 (58F11 81C99 82A99)},
  MRNUMBER = {980178},
MRREVIEWER = {Demetris\ P. K. Ghikas},
       DOI = {10.1007/BF01023872},
       URL = {https://doi-org.ezproxy.lib.uh.edu/10.1007/BF01023872},
}

@incollection {S,
    AUTHOR = {Sinai, Ya. G.},
     TITLE = {The absence of the {P}oisson distribution for spacings between
              quasi-energies in the quantum kicked-rotator model},
      NOTE = {Progress in chaotic dynamics},
   JOURNAL = {Phys. D},
  FJOURNAL = {Physica D. Nonlinear Phenomena},
    VOLUME = {33},
      YEAR = {1988},
    NUMBER = {1-3},
     PAGES = {314--316},
      ISSN = {0167-2789,1872-8022},
   MRCLASS = {28D05 (58F11 81C99 82A99)},
  MRNUMBER = {984625},
MRREVIEWER = {Demetris\ P. K. Ghikas},
       DOI = {10.1016/S0167-2789(98)90024-0},
       URL = {https://doi-org.ezproxy.lib.uh.edu/10.1016/S0167-2789(98)90024-0},
}

@article {RZ,
    AUTHOR = {Rudnick, Z. and Zaharescu, A.},
     TITLE = {The distribution of spacings between fractional parts of
              lacunary sequences},
   JOURNAL = {Forum Math.},
  FJOURNAL = {Forum Mathematicum},
    VOLUME = {14},
      YEAR = {2002},
    NUMBER = {5},
     PAGES = {691--712},
}

@article{MOS,
author = {Mohammadi, A. and Ostafe, A. and Shparlinski, I.},
title = {On some matrix counting problems},
journal = {Journal of the London Mathematical Society},
volume = {110},
number = {6},
pages = {e70044},
year = {2024}
}

@article {BL,
    AUTHOR = {Blomer, V. and Li, J.},
     TITLE = {Correlations of values of random diagonal forms},
   JOURNAL = {Int. Math. Res. Not. IMRN},
  FJOURNAL = {International Mathematics Research Notices. IMRN},
      YEAR = {2023},
    NUMBER = {23},
     PAGES = {20296--20336},
}

@article {RS2024,
    AUTHOR = {Radziwi\l\l \,, M. and Shubin, A.},
     TITLE = {Poissonian pair correlation for {$\alpha n^\theta \bmod 1$}},
   JOURNAL = {Int. Math. Res. Not. IMRN},
  FJOURNAL = {International Mathematics Research Notices. IMRN},
      YEAR = {2024},
    NUMBER = {9},
     PAGES = {7654--7679},
}

@article {ABR,
    AUTHOR = {Aistleitner, C. and Blomer, V. and Radziwi\l\l,
              M.},
     TITLE = {Triple correlation and long gaps in the spectrum of flat tori},
   JOURNAL = {J. Eur. Math. Soc. (JEMS)},
  FJOURNAL = {Journal of the European Mathematical Society (JEMS)},
    VOLUME = {26},
      YEAR = {2024},
    NUMBER = {1},
     PAGES = {41--74},
}

@article {LST,
    AUTHOR = {Lutsko, C. and Sourmelidis, A. and Technau,
              N.},
     TITLE = {Pair correlation of the fractional parts of {$\alpha
              n^{\theta}$}},
   JOURNAL = {J. Eur. Math. Soc. (JEMS)},
  FJOURNAL = {Journal of the European Mathematical Society (JEMS)},
    VOLUME = {27},
      YEAR = {2025},
    NUMBER = {10},
     PAGES = {4069--4082},
}

@article {MR1438823,
    AUTHOR = {Vaughan, R. C. and Wooley, T. D.},
     TITLE = {A special case of {V}inogradov's mean value theorem},
   JOURNAL = {Acta Arith.},
  FJOURNAL = {Acta Arithmetica},
    VOLUME = {79},
      YEAR = {1997},
    NUMBER = {3},
     PAGES = {193--204},
      ISSN = {0065-1036,1730-6264},
   MRCLASS = {11D72 (11L15 11P05)},
  NOOPMRNUMBER = {1438823},
MRREVIEWER = {D.\ R.\ Heath-Brown},
       NOOPDOI = {10.4064/aa-79-3-193-204},
       NOOPURL = {https://doi.org/10.4064/aa-79-3-193-204},
}

@article {MR4345823,
    AUTHOR = {de la Bret\`eche, R. and Kurlberg, P. and Shparlinski,
              I.},
     TITLE = {On the number of products which form perfect powers and
              discriminants of multiquadratic extensions},
   JOURNAL = {Int. Math. Res. Not. IMRN},
  FJOURNAL = {International Mathematics Research Notices. IMRN},
      YEAR = {2021},
    NUMBER = {22},
     PAGES = {17140--17169},
      ISSN = {1073-7928,1687-0247},
   MRCLASS = {11N37 (11N45 11R20)},
  NOOPMRNUMBER = {4345823},
MRREVIEWER = {Frank\ Henry\ Thorne},
       NOOPDOI = {10.1093/imrn/rnz316},
       NOOPURL = {https://doi.org/10.1093/imrn/rnz316},
}

@article {tolev,
    AUTHOR = {Tolev, D. I.},
     TITLE = {On the number of pairs of positive integers {$x_1,x_2\leq H$}
              such that {$x_1x_2$} is a {$k$}-th power},
   JOURNAL = {Pacific J. Math.},
  FJOURNAL = {Pacific Journal of Mathematics},
    VOLUME = {249},
      YEAR = {2011},
    NUMBER = {2},
     PAGES = {495--507},
      ISSN = {0030-8730,1945-5844},
   MRCLASS = {11D45 (11N37)},
  NOOPMRNUMBER = {2782682},
MRREVIEWER = {D.\ R.\ Heath-Brown},
       NOOPDOI = {10.2140/pjm.2011.249.495},
       NOOPURL = {https://doi.org/10.2140/pjm.2011.249.495},
}

@article {LT1,
    AUTHOR = {Lutsko, C. and Technau, N.},
     TITLE = {{$m$}-point correlations of the fractional parts of {$\alpha
              n^\theta$}},
   JOURNAL = {Amer. J. Math.},
  FJOURNAL = {American Journal of Mathematics},
    VOLUME = {148},
      YEAR = {2026},
    NUMBER = {2},
     PAGES = {569--597},
      ISSN = {0002-9327,1080-6377},
   MRCLASS = {11K31 (11L07 42B10)},
  NOOPMRNUMBER = {5054173},
}

@article {LT2,
    AUTHOR = {Lutsko, C. and Technau, N.},
     TITLE = {Full {P}oissonian local statistics of slowly growing
              sequences},
   JOURNAL = {Compos. Math.},
  FJOURNAL = {Compositio Mathematica},
    VOLUME = {161},
      YEAR = {2025},
    NUMBER = {1},
     PAGES = {148--180},
}

@article {AE-BM,
   AUTHOR = {Aistleitner, C. and El-Baz, D. and Munsch, M.},
   TITLE = {A pair correlation problem, and counting lattice points with the zeta function},
   JOURNAL = {Geom. Funct. Anal.},
   PAGES = {483--512},
   YEAR = {2021},
   VOLUME = {31},
   NUMBER = {3},
}

@article {AistleitnerLarcherLewko2017,
    AUTHOR = {Aistleitner, C. and Larcher, G. and Lewko, M.},
     TITLE = {Additive energy and the {H}ausdorff dimension of the
              exceptional set in metric pair correlation problems},
      NOTE = {With an appendix by Jean Bourgain},
   JOURNAL = {Israel J. Math.},
  FJOURNAL = {Israel Journal of Mathematics},
    VOLUME = {222},
      YEAR = {2017},
    NUMBER = {1},
     PAGES = {463--485},
}

@article {HB,
    AUTHOR = {Heath-Brown, D. R.},
     TITLE = {Sums and differences of three {$k$}th powers},
   JOURNAL = {J. Number Theory},
  FJOURNAL = {Journal of Number Theory},
    VOLUME = {129},
      YEAR = {2009},
    NUMBER = {6},
     PAGES = {1579--1594},
      ISSN = {0022-314X,1096-1658},
   MRCLASS = {11D45 (11E76 11G35 11N32)},
  NOOPMRNUMBER = {2521494},
MRREVIEWER = {Rainer\ Dietmann},
       NOOPDOI = {10.1016/j.jnt.2009.01.012},
       NOOPURL = {https://doi.org/10.1016/j.jnt.2009.01.012},
}

@article {Marmon,
    AUTHOR = {Marmon, O.},
     TITLE = {Sums and differences of four {$k$}th powers},
   JOURNAL = {Monatsh. Math.},
  FJOURNAL = {Monatshefte f\"ur Mathematik},
    VOLUME = {164},
      YEAR = {2011},
    NUMBER = {1},
     PAGES = {55--74},
      ISSN = {0026-9255,1436-5081},
   MRCLASS = {11D45 (11D41 11P05 14G05)},
  NOOPMRNUMBER = {2827172},
MRREVIEWER = {D.\ R.\ Heath-Brown},
       NOOPDOI = {10.1007/s00605-010-0248-2},
       NOOPURL = {https://doi.org/10.1007/s00605-010-0248-2},
}

@article {PLMS,
    AUTHOR = {Salberger, P.},
     TITLE = {Counting rational points on projective varieties},
   JOURNAL = {Proc. Lond. Math. Soc. (3)},
  FJOURNAL = {Proceedings of the London Mathematical Society. Third Series},
    VOLUME = {126},
      YEAR = {2023},
    NUMBER = {4},
     PAGES = {1092--1133},
      ISSN = {0024-6115,1460-244X},
   MRCLASS = {11D45 (11D72 11G35 14G05 14G40)},
  NOOPMRNUMBER = {4574827},
MRREVIEWER = {Paul\ M.\ Voutier},
       NOOPDOI = {10.1112/plms.12508},
       NOOPURL = {https://doi.org/10.1112/plms.12508},
}

@article {SW,
    AUTHOR = {Salberger, P. and Wooley, T.D.},
     TITLE = {Rational points on complete intersections of higher degree,
              and mean values of {W}eyl sums},
   JOURNAL = {J. Lond. Math. Soc. (2)},
  FJOURNAL = {Journal of the London Mathematical Society. Second Series},
    VOLUME = {82},
      YEAR = {2010},
    NUMBER = {2},
     PAGES = {317--342},
      ISSN = {0024-6107,1469-7750},
   MRCLASS = {11D45 (11D41 11D72 11G35 11L15 11P55 14G05)},
  NOOPMRNUMBER = {2725042},
MRREVIEWER = {D.\ R.\ Heath-Brown},
       NOOPDOI = {10.1112/jlms/jdq027},
       NOOPURL = {https://doi.org/10.1112/jlms/jdq027},
}

@article {BHB,
    AUTHOR = {Browning, T. D. and Heath-Brown, D. R.},
     TITLE = {The density of rational points on non-singular hypersurfaces.
              {II}},
      NOTE = {With an appendix by J. M. Starr},
   JOURNAL = {Proc. London Math. Soc. (3)},
  FJOURNAL = {Proceedings of the London Mathematical Society. Third Series},
    VOLUME = {93},
      YEAR = {2006},
    NUMBER = {2},
     PAGES = {273--303},
      ISSN = {0024-6115,1460-244X},
   MRCLASS = {11D45 (11G35 14G05)},
  NOOPMRNUMBER = {2251154},
MRREVIEWER = {Clayton\ Petsche},
       NOOPDOI = {10.1112/S0024611506015784},
       NOOPURL = {https://doi.org/10.1112/S0024611506015784},
}

@article{EMM2005,
  title={Quadratic forms of signature (2, 2) and eigenvalue spacings on rectangular 2-tori},
  author={Eskin, A. and Margulis, G. and Mozes, S.},
  journal={Ann. of Math.},
  volume={161},
  number={2},
  pages={679--725},
  year={2005},
  publisher={JSTOR}
}

@article {paucityvmvt,
    AUTHOR = {Wooley, T.D.},
     TITLE = {Paucity problems and some relatives of {V}inogradov's mean
              value theorem},
   JOURNAL = {Math. Proc. Cambridge Philos. Soc.},
  FJOURNAL = {Mathematical Proceedings of the Cambridge Philosophical
              Society},
    VOLUME = {175},
      YEAR = {2023},
    NUMBER = {2},
     PAGES = {327--343},
      ISSN = {0305-0041,1469-8064},
   MRCLASS = {11D45 (11P05)},
  NOOPMRNUMBER = {4623518},
MRREVIEWER = {Boqing\ Xue},
       NOOPDOI = {10.1017/S0305004123000166},
       NOOPURL = {https://doi.org/10.1017/S0305004123000166},
}

@article {globalfields,
    AUTHOR = {Paredes, M. and Sasyk, R.},
     TITLE = {Uniform bounds for the number of rational points on varieties
              over global fields},
   JOURNAL = {Algebra Number Theory},
  FJOURNAL = {Algebra \& Number Theory},
    VOLUME = {16},
      YEAR = {2022},
    NUMBER = {8},
     PAGES = {1941--2000},
      ISSN = {1937-0652,1944-7833},
   MRCLASS = {11D45 (11G35 11G50 14G05)},
  NOOPMRNUMBER = {4516199},
MRREVIEWER = {Davide\ Lombardo},
       NOOPDOI = {10.2140/ant.2022.16.1941},
       NOOPURL = {https://doi.org/10.2140/ant.2022.16.1941},
}

@article {BP,
    AUTHOR = {Bombieri, E. and Pila, J.},
     TITLE = {The number of integral points on arcs and ovals},
   JOURNAL = {Duke Math. J.},
  FJOURNAL = {Duke Mathematical Journal},
    VOLUME = {59},
      YEAR = {1989},
    NUMBER = {2},
     PAGES = {337--357},
      ISSN = {0012-7094,1547-7398},
   MRCLASS = {11P21 (11D99)},
  NOOPMRNUMBER = {1016893},
MRREVIEWER = {Ulrich\ Rausch},
       NOOPDOI = {10.1215/S0012-7094-89-05915-2},
       NOOPURL = {https://doi.org/10.1215/S0012-7094-89-05915-2},
}

@article {salberger2,
    AUTHOR = {Salberger, P.},
     TITLE = {Counting rational points on hypersurfaces of low dimension},
   JOURNAL = {Ann. Sci. \'Ecole Norm. Sup. (4)},
  FJOURNAL = {Annales Scientifiques de l'\'Ecole Normale Sup\'erieure.
              Quatri\`eme S\'erie},
    VOLUME = {38},
      YEAR = {2005},
    NUMBER = {1},
     PAGES = {93--115},
      ISSN = {0012-9593},
   MRCLASS = {14G05 (11G35)},
  NOOPMRNUMBER = {2136483},
MRREVIEWER = {Tam\'as\ Szamuely},
       NOOPDOI = {10.1016/j.ansens.2004.10.005},
       NOOPURL = {https://doi.org/10.1016/j.ansens.2004.10.005},
}

@article {paucitypairs,
    AUTHOR = {Br\"udern, J. and Wooley, T. D.},
     TITLE = {The paucity problem for certain pairs of diagonal equations},
   JOURNAL = {Q. J. Math.},
  FJOURNAL = {The Quarterly Journal of Mathematics},
    VOLUME = {54},
      YEAR = {2003},
    NUMBER = {1},
     PAGES = {41--48},
      ISSN = {0033-5606,1464-3847},
   MRCLASS = {11D45 (11D85 11P05 11P55)},
  NOOPMRNUMBER = {1967068},
MRREVIEWER = {Scott\ T.\ Parsell},
       NOOPDOI = {10.1093/qjmath/54.1.41},
       NOOPURL = {https://doi.org/10.1093/qjmath/54.1.41},
}

@article {paucityskinner,
    AUTHOR = {Skinner, C. M. and Wooley, T. D.},
     TITLE = {On the paucity of non-diagonal solutions in certain diagonal
              {D}iophantine systems},
   JOURNAL = {Quart. J. Math. Oxford Ser. (2)},
  FJOURNAL = {The Quarterly Journal of Mathematics. Oxford. Second Series},
    VOLUME = {48},
      YEAR = {1997},
    NUMBER = {190},
     PAGES = {255--277},
      ISSN = {0033-5606,1464-3847},
   MRCLASS = {11D41 (11G35)},
  NOOPMRNUMBER = {1458583},
MRREVIEWER = {D.\ R.\ Heath-Brown},
       NOOPDOI = {10.1093/qmath/48.2.255},
       NOOPURL = {https://doi.org/10.1093/qmath/48.2.255},
}

@article {paucitytriples,
    AUTHOR = {Br\"udern, J. and Wooley, T. D.},
     TITLE = {A paucity problem for certain triples of diagonal equations},
   JOURNAL = {Bull. Lond. Math. Soc.},
  FJOURNAL = {Bulletin of the London Mathematical Society},
    VOLUME = {54},
      YEAR = {2022},
    NUMBER = {4},
     PAGES = {1396--1412},
      ISSN = {0024-6093,1469-2120},
   MRCLASS = {11D45 (11P05 11P55)},
  NOOPMRNUMBER = {4488314},
MRREVIEWER = {Boqing\ Xue},
       NOOPDOI = {10.1112/blms.12636},
       NOOPURL = {https://doi.org/10.1112/blms.12636},
}

@book {Debarre,
    AUTHOR = {Debarre, O.},
     TITLE = {Higher-dimensional algebraic geometry},
    SERIES = {Universitext},
 PUBLISHER = {Springer-Verlag, New York},
      YEAR = {2001},
     PAGES = {xiv+233},
      ISBN = {0-387-95227-6},
   MRCLASS = {14-02 (14E30 14Jxx)},
  MRNUMBER = {1841091},
MRREVIEWER = {Mark\ Gross},
       DOI = {10.1007/978-1-4757-5406-3},
       URL = {https://doi.org/10.1007/978-1-4757-5406-3},
}

@article{hareMohantyRoginskaya2007,
  title={A general energy formula},
  author={Hare, K. E. and Mohanty, P. and Roginskaya, M.},
  journal={Mathematica Scandinavica},
  pages={29--47},
  year={2007},
  publisher={JSTOR}
}

@article {HBannals,
    AUTHOR = {Heath-Brown, D. R.},
     TITLE = {The density of rational points on curves and surfaces},
   JOURNAL = {Ann. of Math. (2)},
  FJOURNAL = {Annals of Mathematics. Second Series},
    VOLUME = {155},
      YEAR = {2002},
    NUMBER = {2},
     PAGES = {553--595},
      ISSN = {0003-486X,1939-8980},
   MRCLASS = {11G35 (11G50 14G05 14G40)},
  NOOPMRNUMBER = {1906595},
MRREVIEWER = {Carlo\ Gasbarri},
       NOOPDOI = {10.2307/3062125},
       NOOPURL = {https://doi.org/10.2307/3062125},
}

@article {Heath-Brown2010,
    AUTHOR = {Heath-Brown, D. R.},
     TITLE = {Pair correlation for fractional parts of {$\alpha n^2$}},
   JOURNAL = {Math. Proc. Cambridge Philos. Soc.},
    VOLUME = {148},
      YEAR = {2010},
    NUMBER = {3},
     PAGES = {385--407},
}

@article{kerr2025,
  title={Metric {P}oissonian pair correlation for real sequences and energy estimates},
  author={Kerr, B. and Wang, H.},
  journal={Preprint, arXiv:2506.17031},
}

@article{wooleybrudernrobert,
  title={Strong paucity in the {B}r{\"u}dern--{R}obert {D}iophantine system},
  author={T. D. Wooley},
  journal={Preprint, arXiv:2601.05121},
}

@book {KupiersNiederreiter1974,
    AUTHOR = {Kuipers, L. and Niederreiter, H.},
     TITLE = {Uniform distribution of sequences},
      NOTE = {Pure and Applied Mathematics},
 PUBLISHER = {Wiley-Interscience [John Wiley \& Sons], New
              York-London-Sydney},
      YEAR = {1974},
     PAGES = {xiv+390},
}

@InProceedings{Marklof2000,
author="Marklof, J.",
title="The {B}erry-{T}abor Conjecture",
booktitle="Proceedings of the {$3^{rd}$} European Congress of Mathematics",
year="2000",
volume ="202",
publisher="Birkh{\"a}user, Basel",
address="Barcelona",
pages="421-427",
}

@article{Marklof2003,
  title={Pair correlation densities of inhomogeneous quadratic forms},
  author={Marklof, J.},
  journal={Ann. of Math.},
  volume = {158},
  number={2},
  pages={419--471},
  year={2003},
  publisher={JSTOR}
}

@article {R,
    AUTHOR = {Rudnick, Z.},
     TITLE = {What is{$\dots$} quantum chaos?},
   JOURNAL = {Notices Amer. Math. Soc.},
  FJOURNAL = {Notices of the American Mathematical Society},
    VOLUME = {55},
      YEAR = {2008},
    NUMBER = {1},
     PAGES = {32--34},
}

@article {RudnickKurlberg1999,
    AUTHOR = {Kurlberg, P. and Rudnick, Z.},
     TITLE = {The distribution of spacings between quadratic residues},
   JOURNAL = {Duke Math. J.},
  FJOURNAL = {Duke Mathematical Journal},
    VOLUME = {100},
      YEAR = {1999},
}

@article {RS,
    AUTHOR = {Rudnick, Z. and Sarnak, P.},
     TITLE = {The pair correlation function of fractional parts of
              polynomials},
   JOURNAL = {Comm. Math. Phys.},
    VOLUME = {194},
      YEAR = {1998},
    NUMBER = {1},
     PAGES = {61--70},
}

@article {RSZ,
    AUTHOR = {Rudnick, Z. and Sarnak, P. and Zaharescu, A.},
     TITLE = {The distribution of spacings between the fractional parts of
              {$n^2\alpha$}},
  JOURNAL = {Inventiones Mathematicae},
    VOLUME = {145},
      YEAR = {2001},
    NUMBER = {1},
     PAGES = {37--57},
}

@Article{RudnickTechnau2021,
author="Rudnick, Z. and Technau, N.",
title="The metric theory of the pair correlation function for small non-integer powers
",
year = "2022",
journal="J. Lond. Math. Soc",
volume = {106},
number = {3},
pages = {2752-2772},
}

@article{TechnauYesha2020,
  title={On the correlations of {$\alpha n^\theta$} mod 1},
  author={Technau, N. and Yesha, N.},
  journal={J. Eur. Math. Soc. (JEMS)},
  volume = {25}, 
  number = {10},
  pages  = {4123–4154},
  year   = {2020},
}

@article {TechnauWalker2020,
    AUTHOR = {Technau, N. and Walker, A.},
     TITLE = {On the triple correlations of fractional parts of {$n^2\alpha
              $}},
   JOURNAL = {Canad. J. Math.},
  FJOURNAL = {Canadian Journal of Mathematics. Journal Canadien de
              Math\'ematiques},
    VOLUME = {74},
      YEAR = {2022},
    NUMBER = {5},
     PAGES = {1347--1384},
      ISSN = {0008-414X,1496-4279},
   MRCLASS = {11K06 (11K60 11L05)},
  NOOPMRNUMBER = {4504666},
MRREVIEWER = {Stefan\ Steinerberger},
       NOOPDOI = {10.4153/S0008414X21000249},
       NOOPURL = {https://doi.org/10.4153/S0008414X21000249},
}

@article{vanderkam2000,
  title={Correlations of Eigenvalues on 
  Multi-Dimensional Flat Tori},
  author={Vanderkam, J. M.},
  journal={Communications in Mathematical Physics},
  volume={210},
  number={1},
  pages={203--223},
  year={2000},
  publisher={Springer}
}

@article{verzobio2025,
  title={Counting rational points on smooth hypersurfaces with high degree},
  author={Verzobio, M.},
  journal={International Mathematics Research Notices},
  volume={2025},
  number={16},
  pages={rnaf249},
  year={2025},
  publisher={Oxford University Press}
}

    \hrulefill

    \vspace{4mm}
     \noindent Department of Mathematics, University of Houston, 3551 Cullen Blvd, 77204, Houston, Texas, United States.\\
     \emph{E-mail: \textbf{clutsko@uh.edu}}\\
     \\
     \noindent Institute of Analysis and Number Theory,
    TU Graz,
    Steyrergasse 30/II,
    8010 Graz, 
    Austria.\\
    \emph{E-mail:
    \textbf{rome@tugraz.at}}\\
    \\
    \noindent Department of Mathematics,
    480 Lincoln Drive,
    Madison, WI 53706, 
    USA.\\
    \emph{E-mail: 
    \textbf{technau@wisc.edu}}
    
\end{document}